\documentclass[12pt,a4paper]{article}
\usepackage{amsmath}
\usepackage{amsthm}
\usepackage{amssymb}
\usepackage{amsbsy}
\usepackage{amsfonts}
\usepackage{amstext}
\usepackage{amscd}
\usepackage[dvips]{epsfig}

\newtheorem{theorem}{Theorem}[section]

\newtheorem{proposition}[theorem]{Proposition}
\newtheorem{corollary}[theorem]{Corollary}

\theoremstyle{definition}
\newtheorem{definition}[theorem]{Definition}

\theoremstyle{remark}

\numberwithin{equation}{section}

\begin{document}
\title{Classical and free infinite divisibility for Boolean stable laws}

\author{Octavio Arizmendi\footnote{
OA: Universit\"{a}t des Saarlandes, FR $6.1-$Mathematik, 66123 Saarbr\"{u}cken, Germany, email: arizmendi@math.uni-sb.de. Supported by DFG-Deutsche Forschungsgemeinschaft Project SP419/8-1} and Takahiro Hasebe\footnote{TH: Graduate School of Science,  Kyoto University,  Kyoto 606-8502, Japan, email: thasebe@math.kyoto-u.ac.jp. Supported by Global COE program ``Fostering top leaders in mathematics---broadening the core and exploring new ground'' at Kyoto university.}}
\date{}

\maketitle 
%    Abstract is required.
\begin{abstract}
We completely determine the free infinite divisibility for the Boolean stable law which is parametrized by a stability index $\alpha$ and an asymmetry coefficient $\rho$. We prove that the Boolean stable law is freely infinitely divisible if and only if one of the following conditions holds: $0<\alpha\leq\frac{1}{2}$; $\frac{1}{2}<\alpha\leq\frac{2}{3}$ and $2-\frac{1}{\alpha}\leq\rho \leq \frac{1}{\alpha}-1$;  
 $\alpha=1,~\rho=\frac{1}{2}$. 
Positive Boolean stable laws corresponding to $\rho =1$ and $\alpha \leq \frac{1}{2}$ have completely monotonic densities and they are both freely and classically infinitely divisible. 
We also show that continuous Boolean convolutions of positive Boolean stable laws with different stability indices are also  freely and classically infinitely divisible.
Boolean stable laws, free stable laws and continuous Boolean convolutions of positive Boolean stable laws  are non-trivial examples whose free divisibility indicators are infinity.  We also find that the free multiplicative convolution of Boolean stable laws is again a Boolean stable law. 
%Finally we point out that the Boolean stable law appears as the quotient of classical stable random variables and that of free stable ones. 
\end{abstract}

%Keywords: Free convolution, Boolean stable law, free divisibility indicator, completely monotonic function

\section{Introduction}

In probability theory, infinitely divisible distributions play a central role since they arise as the limit distributions of very general limit 
theorems. In free probability theory \cite{Voi86,Voi87} there is an analogous notion of infinite divisibility. Moreover, the so-called Bercovici-Pata bijection \cite{BePa} maps classically infinitely divisible probability measures to freely infinitely divisible ones. Let $ID(\ast)$ be the class of all classically infinitely divisible distributions on $\mathbb{R}$ and $ID(\boxplus)$ be the class of all freely infinitely divisible distributions on $\mathbb{R}$. 

Classically and freely infinitely divisible probability measures are typically quite different. For example, measures in $ID(\boxplus)$ cannot have more than one atom, while there are many examples of purely atomic measures in $ID(\ast)$. Also, there are many measures in  $ID(\boxplus)$ with compact support, while every non trivial measure in $ID(\ast)$ has unbounded support. So one might expect that, perhaps with very particular exceptions, some probability measures belong to the ``classical world'' and some to the ``free world''. This paper, however, discovers a family of measures lying on the intersection of $ID(\boxplus)$ and $ID(\ast)$. More interestingly, the examples considered
here come from the ``Boolean world'', another framework of non-commutative probability.   

The Boolean convolution was defined in \cite{S-W} and Boolean stable laws, denoted by $\mathbf{b}_\alpha^\rho$, were classified. Here $\alpha$ is the stability index and $\rho$ is an asymmetry coefficient (see Section \ref{prel} below). 

In this paper we consider the classical and free infinite divisibility of Boolean stable laws. On the free probability side, we completely determine which Boolean stable laws are freely infinitely divisible.

\begin{theorem} \label{T1}
The Boolean stable law $\mathbf{b}^\rho_\alpha$ is freely infinitely divisible if and only if:

(1) $0<\alpha\leq\frac{1}{2}$, 
 
(2) $\frac{1}{2}<\alpha\leq\frac{2}{3}$,  $2-\frac{1}{\alpha}\leq\rho \leq \frac{1}{\alpha}-1$, 

(3) $\alpha=1,~\rho=\frac{1}{2}$. 

%Moreover in the case (1) and (3) $\mathbf{b}_\alpha\in\mathcal{UI}$, but $\mathbf{b}_\alpha\notin\mathcal{UI}$, in the case (2)
\end{theorem}

We use analytical methods to prove Theorem \ref{T1} but also show a reproducing property of Boolean stable laws from which one can give a much simpler proof of part \emph{(1)} in Theorem  \ref{T1}, when $\alpha=1/n$, 
for $n$ an integer greater than $1$.

Boolean stable laws not only provide new examples of measures which are infinitely divisible in the free sense but also in the classical one. More explicitly, we prove the following.

\begin{theorem}\label{T3} Boolean stable laws on $[0,\infty)$ are $*$-infinitely divisible for $\alpha \leq 1/2$. 
\end{theorem}

Until now, there were only $5$ known examples of measures in $ID(\boxplus)\cap ID(\ast)$: Cauchy distribution, free $1/2$-stable law, Gaussian, chi-square with one degree of freedom, $F$-distribution with 1,2 degrees of freedom~\cite{AHS, BBLS, PAS}. Theorems \ref{T1} and \ref{T3} give a family of probability measures which are infinitely divisible with respect to both classical and free convolutions. Moreover, continuous Boolean convolutions of Boolean stable laws also belong to $ID(\boxplus)\cap ID(\ast)$ as we prove in Theorems \ref{mix1} and \ref{mix2} below.

Theorem \ref{T1} has another consequence regarding the free divisibility indicator of Belinschi and Nica \cite{BN} (see Section \ref{prel} for a brief explanation on this indicator). Recently,  we were able to prove in \cite{AH2} that any Boolean stable law  $\mathbf{b}_\alpha^\rho$ has free divisibility indicator $0$ or $\infty$, depending on whether or not $\mathbf{b}_\alpha^\rho$ is in $ID(\boxplus)$. The first non trivial example of a measure with divisibility indicator $\infty$ was found from this observation: free and Boolean stable laws with parameter $1/2$. Theorem \ref{T1} broadly generalizes the result of \cite{AH2}. 

Apart from this introduction this note is organized as follows. Section \ref{prel} presents basic preliminaries needed in this paper. In Section \ref{S3} we consider the free infinite divisibility of Boolean stable laws and continuous Boolean convolutions of them. In Section \ref{S4}, we prove the classical infinite divisibility of some Boolean stable laws. Finally, Section \ref{S2} is on a reproducing property of Boolean stable laws with respect to the free multiplicative convolution. 
%It is found that the positive Boolean stable law appears as the quotient of two independent classical stable random variables, and also as the quotient of free stable random variables. 
 
\section{Preliminaries}\label{prel}

\subsection*{Boolean convolution and Boolean stable laws}

We denote by $\mathcal{P}$ the set of Borel probability measures on $\mathbb{R}$. The upper half-plane and the lower half-plane are respectively denoted by $\mathbb{C}^+$ and $\mathbb{C}^-$. 
Let $G_\mu(z) = \int_{\mathbb{R}}\frac{\mu(dx)}{z-x}$ $(z \in \mathbb{C}^+)$ be the Cauchy transform of $\mu \in \mathcal{P}$ and 
$F_\mu(z)$ its reciprocal $\frac{1}{G_\mu(z)}$. The dilation operator $D_b$ $(b>0)$ is defined  on $\mathcal{P}$ so that 
$$
\int_{\mathbb{R}}f(x)D_b (\mu)(dx)=\int_{\mathbb{R}}f(bx)\mu(dx)
$$
for any bounded, continuous function $f$ on $\mathbb{R}$.

The \textit{Boolean convolution} $\uplus$ was defined in \cite{S-W} as follows.
Let $K_\mu(z)$ be the energy function  defined by 
\[
K_\mu(z)=z-F_\mu(z),~~z\in \mathbb{C}^+ 
\]
for $\mu \in \mathcal{P}$. 
The Boolean convolution $\mu \uplus \nu$ is characterized by $K_{\mu \uplus \nu} (z)=K_\mu(z)+K_\nu(z)$. 
%Any probability measure is infinitely divisible with respect to the Boolean convolution and the L\'{e}vy-Khinchin formula is written as \cite{S-W}
%\[
%K_\mu(z) = \gamma_\mu +\int_{\mathbb{R}}\frac{1+xz}{z-x}\tau_\mu(dx),  
%\]
%where $\gamma_\mu \in \mathbb{R}$ and $\tau_\mu$ is a non-negative finite measure. 

A Boolean convolution semigroup $\{\mu^{\uplus t}\}_{t \geq 0}$ can always be defined for any probability measure $\mu$, corresponding to the energy transform $tK_\mu$. 
%\[
%K_{\mu^{\uplus t}}(z) = tK_\mu(z).  
%\]
%We note that $\eta_\mu$ and $K_\mu$ are related with the formula $\eta_\mu(z) = z K_\mu(\frac{1}{z})$. 

%Let $k_\mu(z) := \frac{z}{\eta_\mu(z)}$.

%Similar calculations as for the free additive convolution show the following. 
%\begin{equation}\label{mult-additiveb}
%(\mu\boxtimes\nu)^{\uplus t}=D_t(\mu^{\uplus t}\boxtimes\nu^{\uplus t})~~~~t>1.
%\end{equation}

Boolean stable laws can be defined in terms of self-similarity. We include only non trivial measures, that is, measures which are different from a point measure $\delta_a$. 

\begin{definition} 1) A non trivial measure $\mu$ is said to be \textit{Boolean (or $\uplus$-) strictly stable} if $\mu \uplus \mu=D_{2^{1/\alpha}}(\mu)$ for some $\alpha$. We call $\alpha$ the index of stability. 
   
2)  A non trivial measure is said to be \textit{$\uplus$-stable} if  $\mu\uplus\delta_a$ is $\uplus$-strictly stable for some $a \in \mathbb{R}$.
\end{definition}
 
Boolean strictly stable distributions were characterized in \cite{S-W} in terms of the energy function. General $\uplus$-stable distributions can be classified as follows. 

\begin{theorem} 
Probability measures $\{{\normalfont\textbf{b}}_{\alpha}^\rho: \alpha \in (0,2], \rho \in [0,1]\}$ defined by the following energy functions are $\uplus$-stable: 

i) $K(z)= - e^{i\pi \rho\alpha}z^{-\alpha+1}$ for $\alpha\in(0,1)$ and $\rho \in [0,1]$; 

ii) $K(z)= -2\rho i +\frac{2(2\rho-1)}{\pi}\log z$ for $\alpha =1$ and $\rho \in [0,1]$; 

iii)  $K(z)= e^{i(\alpha-2)\rho\pi} z^{-\alpha+1}$ for $\alpha\in(1,2]$ and $\rho \in [0,1]$. 

The index of stability of ${\normalfont\textbf{b}}_{\alpha}^\rho$ is $\alpha$. 
For any $\uplus$-stable law $\mu$, there exist $a \in \mathbb{R}, b >0, \alpha \in(0,2], \rho\in [0,1]$ such that  $\mu= \delta_a \uplus D_b({\normalfont\textbf{b}}_{\alpha}^\rho)$. 
\end{theorem}
As shown in \cite{AH2}, the transformation $\mu \mapsto \delta_a \uplus D_b(\mu)$ does not change the free infinite divisibility (free divisibility indicator, more strongly), so that 
the above cases i)--iii) are enough for our purpose. 

The parameter $\rho$ is called an \textit{asymmetry coefficient} of a $\uplus$-stable law. 
The measure ${\normalfont\textbf{b}}_{\alpha}^\rho$ is supported on $[0,\infty)$ if and only if $\rho =1$, $\alpha \in (0,1]$ and it is symmetric  if and only if $\rho=\frac{1}{2}$. 
 
%Closely related to the notion of stability is that of domains of attraction. Recall that for a probability measure $\mu$ we say that $\nu$ is in the Boolean domain of attraction of $\nu$ if there exists $\alpha$ such that $D_{N^{\alpha}}(\mu^{\uplus N})\rightarrow \nu_i$. The following theorem explains the relation between domains of attraction and strictly stable laws.

%\begin{theorem}\cite{BePa}
%A measure is $\mu$ is $\uplus$-stable laws if and only its the Boolean domain of attraction is not empty.
%\end{theorem}

\subsection*{Free additive convolution}
The \emph{free additive convolution} $\boxplus$ was introduced by Voiculescu in \cite{Voi86} for compactly supported measures on $\mathbb{R}$ and is defined as follows. 
If $X_{1},X_{2}$ are free random variables following probability distributions $\mu_1, \mu_2$ respectively, then the probability distribution of $X_1+X_2$ is denoted by $\mu_1 \boxplus \mu_2$ and is called the free additive convolution. The free additive convolution was later extended to all probability measures in \cite{Be-Vo} and it is characterized as follows. 

For any $\beta>0$, there exists $M>0$ such that the reciprocal Cauchy transform $F_\mu$ has a right inverse map $F^{-1}_\mu$ defined in 
$\Gamma_{\beta, M}:=\{z \in \mathbb{C}^+: \text{Im}\,z >M,~ \beta|\text{Re}\,z| <\text{Im}\,z\}$. 
Let $\phi_\mu(z)$ be the Voiculescu transform of $\mu$ defined by 
\[
\phi_\mu(z) = F_\mu^{-1}(z)-z,~~z\in\Gamma_{\beta,M}.  
\]
The free convolution $\mu \boxplus \nu$ of probability measures $\mu$ and $\nu$ is the unique probability measure such that $\phi_{\mu \boxplus \nu}(z) = \phi_\mu (z) + \phi_\nu(z)$ in a common domain.

\subsection*{Free infinite divisibility and divisibility indicator}\label{S9}
\begin{definition} A probability measure $\mu$ is said to be \textit{freely (or $\boxplus$-) infinitely divisible}  if for each $n\in\mathbb{N}$ there exist $\mu_n \in \mathcal{P}$ such that $\mu=\mu_n^{\boxplus n}.$ We will denote by $ID(\boxplus)$ the set of $\boxplus$-infinitely divisible measures. 
\end{definition}

It is known that a probability measure $\mu$ belongs to $ID(\boxplus)$ if and only if $\phi_\mu$ extends analytically to $\mathbb{C}^+$ with values in $\mathbb{C}^- \cup \mathbb{R}$. The free infinite divisibility of $\mu$ is also equivalent to the existence of probability measures $\mu^{\boxplus t}$ for $0 \leq t < \infty$ which satisfy 
$\phi_{\mu^{\boxplus t}}(z) = t\phi_\mu(z)$~\cite{Be-Vo},  while the partial semigroup $\{\mu^{\boxplus t}\}_{t\geq 1}$ always exists for any $\mu \in \mathcal{P}$~\cite{NS97}. 

%If $\mu$ is $\boxplus$-infinitely divisible, then it has the L\'{e}vy-Khinchin formula 
%\[
%\phi_\mu(z) = \gamma_\mu +\int_{\mathbb{R}}\frac{1+xz}{z-x}\tau_\mu(dx), ~~z\in \mathbb{C}^+,  
%\]
%where $\gamma_\mu \in \mathbb{R}$ and $\tau_\mu$ is a non-negative finite measure. 

A particularly important subset of   $ID(\boxplus)$ is the class of free regular measures.
\begin{definition}
\label{characterization regular} A probability measure $\mu $ is \textit{free regular} if $\mu ^{\boxplus t}$ exists as a probability measure on $[0,\infty)$ for all $t>0.$
\end{definition}
The importance of this class is based on the following properties: this class coincides with the distributions of free subordinators;  this class is closed under free multiplicative convolution~\cite{AHS}. Moreover, 
\begin{theorem}\label{closure}
Let $\mu $ be free regular and $\sigma$ be freely infinitely divisible. 
Then $\mu \boxtimes \sigma $ is freely infinitely divisible.
\end{theorem}

Recall that the semigroup $\{\mathbb{B}_t \}_{t \geq 0}$, introduced by Belinschi and Nica \cite{BN}, is defined to be 
\[
\mathbb{B}_t(\mu) = \Big(\mu^{\boxplus (1+t)}\Big) ^{\uplus\frac{1}{1+t}},~~\mu\in\mathcal{P}. 
\]
$\mathbb{B}_1$ coincides with the Bercovici-Pata bijection $\Lambda_B$ from the Boolean convolution to the free one. The reader is referred to \cite{BePa} for the definition of $\Lambda_B$. Let $\phi(\mu)$ denote the free divisibility indicator defined by 
\[
\phi(\mu):=\sup \{t \geq 0: \mu \in \mathbb{B}_t(\mathcal{P}) \}. 
\]
A probability measure $\mu$ is $\boxplus$-infinitely divisible if and only if $\phi(\mu) \geq 1$. Moreover $\mu^{\boxplus t}$ exists for $t \geq \max\{1-\phi(\mu),0 \}$.

The following property was proved in \cite{AH2}.
\begin{proposition}\label{prop08} Let $\mu$ be a probability measure on $\mathbb{R}$. Then $\phi(\mu^{\uplus t}) = \frac{1}{t}\phi(\mu)$ for $t > 0$. In particular, $\phi(\mu) = \sup\{ t >0: \mu^{\uplus t} \in ID(\boxplus)\}$. 
\end{proposition}

As a consequence, any Boolean stable law $\mathbf{b}_\alpha^\rho$ has free divisibility indicator $0$ or $\infty$, depending on whether or not $\mathbf{b}_\alpha^\rho$ is in $ID(\boxplus)$.   From this observation the first non-trivial example of measures with  $\phi(\mu) =\infty$ was found.  
\begin{proposition}\cite{AH2} \label{T2}
The measure $\mathbf{b}_{1/2}^{\rho}$ is $\boxplus$-infinitely divisible for any $\rho \in [0,1]$. In particular, $\phi(\mathbf{b}_{1/2}^{\rho})=\infty$. Moreover, $\mathbf{b}^1_{1/2}$ is free regular. 
\end{proposition}

We will generalize this proposition in the following sections.

\section{Free infinite divisibility of Boolean stable laws}\label{S3}

In this section we completely solve the problem of the free infinite divisibility of the Boolean stable law, that is, we prove Theorem \ref{T1}.  It suffices to consider $\textbf{b}_\alpha^\rho$ because the free divisibility indicator and hence free infinite divisibility is invariant with respect to the dilation and Boolean convolution with a point measure (see Proposition 3.7 of \cite{AH2}).    

To prove the free infinite divisibility of some Boolean stable laws, the following subclass of $\boxplus$-infinitely divisible distributions will be useful. 
\begin{definition}
A probability measure $\mu$ is said to be in class $\mathcal{UI}$ if $F_\mu$ is univalent in $\mathbb{C}^+$ and, moreover, $F_\mu^{-1}$ has an analytic continuation from $F_\mu(\mathbb{C}^+)$ to $\mathbb{C}^+$ as a univalent function. 
\end{definition}
The following property was used in \cite{BBLS}. See also \cite{AH1} for applications. 
 
\begin{proposition}\label{lem1}
$\mu \in \mathcal{UI}$ implies that $\mu$ is $\boxplus$-infinitely divisible. Moreover, $\mathcal{UI}$ is closed with respect to the weak convergence. 
\end{proposition}

We now go on to the main results. 

\begin{proposition}\label{prop90}
Let  $0<\alpha\leq\frac{1}{2}$, then $\mathbf{b}_\alpha^\rho$ belongs to $\mathcal{UI}$ for any $\rho \in [0,1]$. 
\end{proposition}
In the proof, the following fact is useful. 
\begin{proposition}\label{prop91}
Let $\gamma \subset \mathbb{C}$ be a Jordan curve and $J_\gamma$ be the bounded open set surrounded by $\gamma$. Let $f: J_\gamma \to \mathbb{C}$ be an analytic map which extends to a continuous map from $\overline{J_\gamma}$ to $\mathbb{C}$. If $f$ is injective on $\gamma$, then $f$ is injective in $\overline{J_\gamma}$ and $f(J_\gamma)$ is a bounded Jordan domain surrounded by the Jordan curve $f(\gamma)$. 
\end{proposition}
For the proof, the reader is referred to \cite{Buc}, pp.\ 310, where the above fact is proved for $\gamma = \{z\in\mathbb{C}: |z|=1\}$. The general case follows from the Carath\'eodory theorem (or Osgood-Taylor-Carath\'eodory theorem), see \cite{Buc}, pp.\ 309.

\begin{proof}[Proof of Proposition \ref{prop90}] 

We introduce some notations: $l_\theta$ will denote the half line $\{re^{i\theta}:r \geq 0\}$; $\mathbb{C}_{A}$ will denote the region $\{z \in \mathbb{C}\setminus\{0\}: \arg z \in A \}$ for $A \subset \mathbb{R}$, $\sup\{|x-y|: x,y\in A\} \leq 2\pi$.   
 
Assume that $\alpha \in(0,\frac{1}{2})$. The case $\alpha =\frac{1}{2}$ follows from the fact that $\mathcal{UI}$ is closed under the weak convergence. Let $\phi:=\alpha\rho\pi \in[0,\alpha \pi]$ and $F(z):=z+e^{i\phi}z^{1-\alpha}$.  A key to the proof is the angles 
$$
\theta_1:= -\frac{\phi}{1-\alpha},~~~~\theta_2:=\frac{\pi-\phi}{1-\alpha}. 
$$
Since $\phi \in [0,\alpha \pi]$, it holds that $\theta_1 \in (-\pi, 0]$ and $\theta_2 \in [\pi, 2\pi)$. Also we can see $\pi < \theta_2 - \theta_1 < 2\pi$. It holds that $F(l_{\theta_1}\cup l_{\theta_2})\subset\mathbb{C}^-\cup\mathbb{R}$ since  
$$
\text{Im}\, F(re^{i\theta_1}) = r \sin\theta_1 \leq 0,~~~~\text{Im}\, F(re^{i\theta_2}) = r \sin\theta_2 \leq 0,~~~r \geq 0. 
$$
It is sufficient to prove that $F$ is injective on $l_{\theta_1} \cup l_{\theta_2}$ for the following reason. Since $F(z) =z(1+o(1))$ uniformly as $|z| \to \infty$, $z \in \mathbb{C}_{[\theta_1,\theta_2]}$, $F$ is injective on $\{z \in \mathbb{C}: \arg z \in [\theta_1,\theta_2], |z| >R\}$ for sufficiently large $R>0$.  
Therefore, if $F$ is injective on $l_{\theta_1} \cup l_{\theta_2}$, it is also injective on the curve $\Gamma_R:= \{re^{i\theta_1}: 0\leq r \leq R \}\cup \{re^{i\theta_2}: 0\leq r \leq R \} \cup \{Re^{i\theta}: \theta \in [\theta_1,\theta_2] \}$. Now Proposition \ref{prop91} implies that $F$ is injective in the bounded domain  surrounded by $\Gamma_R$ for any large $R>0$. Since $F(l_{\theta_1} \cup l_{\theta_2}) \subset \mathbb{C}^-\cup \mathbb{R}$ and $F(z) = z(1+o(1))$ as $|z|\to \infty$, each point of $\mathbb{C}^+$ is surrounded by $F(\Gamma_R)$ exactly once for large $R>0$. This means that  $F$ is injective in $\mathbb{C}_{(\theta_1,\theta_2)}$ and $F(\mathbb{C}_{(\theta_1,\theta_2)}) \supset \mathbb{C}^+$ and hence $F^{-1}$ analytically extends to $\mathbb{C}^+$ as a univalent map. 
 
 Now let us prove that  $F$ is injective on $l_{\theta_1} \cup l_{\theta_2}$. 
It holds that $F(re^{i\theta_1}) = re^{i\theta_1} + r^{1-\alpha} \in \mathbb{C}_{[\theta_1,0]}$ and similarly $F(re^{i\theta_2}) \in \mathbb{C}_{[\pi,\theta_2]}$ for any $r >0$. Since $0< \theta_2 -\theta_1 < 2\pi$, the intersection $\mathbb{C}_{[\theta_1,0]}\cap \mathbb{C}_{[\pi,\theta_2]}$ is empty. This means $F(l_{\theta_1}) \cap F(l_{\theta_2}) = \{0\}$. 

Therefore, we only have to prove that $F$ is injective on each $l_{\theta_k}$. 
Suppose $z,w \in l_{\theta_1}$ satisfy both $z\neq w$ and $F(z)=F(w)$. With the notation $z = re^{i\theta_1}$ and $w = Re^{i\theta_1}$,  it holds that 
\begin{equation}\label{eq90}
e^{i\theta_1}=- \frac{R^{1-\alpha} -r^{1-\alpha}}{R-r} <0, 
\end{equation}
a contradiction. Hence $F$ is injective on $l_{\theta_1}$. Similarly $F$ is injective on $l_{\theta_2}$. 
\end{proof}

A similar method is applicable to a continuous Boolean convolution  of Boolean stable laws defined by 
$$
F_{\textbf{b}(\sigma)}(z) = z + \int_{(0,\frac{1}{2}]} e^{i\alpha \pi}z^{1-\alpha} \sigma(d\alpha), 
$$ 
where $\sigma$ is a positive finite measure on $(0,\frac{1}{2}]$. Symbolically one may write 
$$
\textbf{b}(\sigma) = \int^{\uplus}_{(0,\frac{1}{2}]} \textbf{b}_\alpha^{1} \sigma(d\alpha). 
$$
The function $F_{\textbf{b}(\sigma)}$ maps $\mathbb{C}^+$ into itself and satisfies $\lim_{y \to \infty}\frac{F_{\textbf{b}(\sigma)}(iy)}{iy} =1$, so that $F_{\textbf{b}(\sigma)}$ indeed defines a probability measure $\textbf{b}(\sigma)$. 

\begin{theorem}\label{mix1}
The measure ${\normalfont \textbf{b}}(\sigma)$ is in $\mathcal{UI}$. Moreover, $\phi({\normalfont\textbf{b}}(\sigma)) = \infty$ and $\mathbf{b}(\sigma)$ is free regular. 
\end{theorem}
\begin{proof}
We follow the notation of Proposition \ref{prop90}. 
Let $F(z):=z + \int_{(0,\frac{1}{2}]} e^{i\alpha \pi}z^{1-\alpha} \sigma(d\alpha)$. 
We may assume that $\sigma$ is supported on $(0, \beta]$, where $\beta \in (0,\frac{1}{2})$;  
a general $\sigma$ can be approximated by such measures. Let $\theta_\beta:=-\frac{\beta\pi}{1-\beta} \in (-\pi,0)$. 
Then the proof of Proposition \ref{prop90} holds with the following modifications. 
From direct calculation, we can prove that $F(l_{\theta_\beta}\cup l_{\pi}) \subset \mathbb{C}^- \cup \mathbb{R}$ and $F(l_{\theta_\beta}) \cap F(l_\pi) = \{0\}$. 
Assume that $z=re^{i\theta_\beta} ,w=Re^{i\theta_\beta}$ satisfy $z \neq w$ and $F(z) =F(w)$, then
$$
e^{i\theta_\beta} =  \int_{(0,\beta]}e^{i((\alpha+1)\pi +(1-\alpha) \theta_\beta)}\frac{R^{1-\alpha}-r^{1-\alpha}}{R-r} \sigma(d\alpha). 
$$
Since $(\alpha+1)\pi +(1-\alpha) \theta_\beta=\pi-\frac{\beta-\alpha}{1-\beta}\pi \in (\theta_\beta+\pi, \pi]$, the above relation causes a contradiction and this proves the injectivity on $l_{\theta_\beta}$. Similarly, $F$ is injective on $l_\pi$. 

The property $\phi(\mathbf{b}(\sigma)) =\infty$ follows from Proposition \ref{prop08} in combination with the property $\textbf{b}(\sigma)^{\uplus t} = \textbf{b}(t\sigma) \in ID(\boxplus)$.  
The free regularity is a consequence of Theorem 13 \cite{AHS}. 
\end{proof}

\begin{proposition}
Let  $\frac{1}{2}<\alpha\leq\frac{2}{3}$ and $2-\frac{1}{\alpha} \leq \rho \leq \frac{1}{\alpha}-1$, then
 $\mathbf{b}_\alpha^\rho$ is freely infinitely divisible, but $\mathbf{b}_\alpha^\rho \notin\mathcal{UI}$.
\end{proposition}
\begin{proof}
Let $\phi:= \alpha \rho\pi$ and $F=F_{\mathbf{b}_\alpha^\rho}$. First we assume that $(2\alpha-1)\pi < \phi < (1-\alpha)\pi$. This inequality implies $\theta_1 \in (-\pi,0)$ and $\theta_2 \in (\pi,2\pi)$. 
 We will prove that $\phi(z)=F^{-1}(z)-z$ has an analytic continuation to $\mathbb{C}^+$ with values in $\mathbb{C}^- \cup \mathbb{R}$. 

We follow the notation of Proposition \ref{prop90}, but for simplicity, let $U_1:= \mathbb{C}_{(\theta_1,\pi)}$ and $U_2:=\mathbb{C}_{(0,\theta_2)}$. In this case, $\theta_2-\theta_1 > 2\pi$, so that we have to consider a Riemannian surface made of $U_1$ and $U_2$, glued on $\mathbb{C}^+$. Let us consider the analytic maps $F_i:=F|_{U_i}$ for $i=1,2$. 
 
Following the arguments of Proposition \ref{prop90}, we can prove that $F_i$ is univalent in $U_i$ for $i=1,2$. For example, let us consider $i=1$. From simple observations, we find $F_1(re^{i\theta_1}) = re^{i\theta_1} + r^{1-\alpha} \in \mathbb{C}_{(\theta_1,0)}$ and $F_1(-r) =-r + r^{1-\alpha} e^{i(\phi+(1-\alpha)\pi)} \in \mathbb{C}_{(\phi + (1-\alpha)\pi,\pi)}$ for any $r>0$. Hence $F(l_{\theta_1}) \cap F(l_{\pi}) = \{0\}$. 

If $z = re^{i\theta_1}, w=Re^{i\theta_1}$ satisfy $z \neq w$, $F_1(z) =F_1(w)$, then (\ref{eq90}) holds, which contradicts the property $\theta_1 \in (-\pi,0)$. If $z,w \in l_\pi$, $z \neq w$, then $e^{i(\phi+(1-\alpha)\pi)} = \frac{R-r}{R^{1-\alpha}-r^{1-\alpha}} >0$, a contradiction. The above arguments imply $F_1$ is univalent in $U_1$ and $F_1(U_1) \supset \mathbb{C}_{(0, \phi+(1-\alpha)\pi)}$. Similarly one can show that $F_2$ is univalent in $U_2$ and $F_2(U_2) \supset \mathbb{C}_{(\phi,\pi)}$. Notice that 
$F_1(U_1) \cup F_2(U_2) \supset \mathbb{C}^+$. These properties still hold for $\phi\in\{(2\alpha-1)\pi, (1-\alpha)\pi\}$ by taking limits. 

The map $F^{-1}$ extends analytically to $\mathbb{C}^+$ as follows:  
\[ 
F^{-1}(z):= \left \{ \begin{array}{cc}
F^{-1}_1(z),       &   z\in F_1(U_1)\cap \mathbb{C}^+,       \\
 F^{-1}_2(z), &    z\in F_2(U_2)\cap \mathbb{C}^+.   
  \end{array}\right.
 \]
This map is well defined for the following reason. 
Notice that $F_1(U_1\setminus \mathbb{C}^+)\cap F_2(U_2\setminus \mathbb{C}^+) \subset \mathbb{C}^-$. Indeed,
$F_1(U_1\setminus \mathbb{C}^+)\subset\mathbb{C}_{(\theta_1, \phi)}$ and $F_2(U_2\setminus \mathbb{C}^+)\subset\mathbb{C}_{(\phi+(1-\alpha)\pi, \theta_2)}$. If $z\in F_1(U_1)\cap F_2(U_2)\cap \mathbb{C}^+$, then there are $z_1\in U_1$ and $z_2\in U_2$ with $z=F_1(z_1)=F_2(z_2)$.
Either $z_1$ or $z_2$ shall be in $\mathbb{C}^+$ because $F_1(U_1\setminus \mathbb{C}^+)\cap F_2(U_2\setminus \mathbb{C}^+) \subset \mathbb{C}^-$. If $z_1\in\mathbb{C}^+$(resp.\ $z_2\in\mathbb{C}^+$), by definition $F_1(z_1)=F_2(z_1)$(resp.\ $F_1(z_2)=F_2(z_2)$) and hence $z_1=z_2$ from the univalence of $F_2$. 

Now, we see that $\phi(z):=F^{-1}(z)-z$ maps $\mathbb{C}^+$ into $\mathbb{C}^- \cup \mathbb{R}$. If $z\in F(\mathbb{C}^+)$, then there is $w\in\mathbb{C}^+$ such that $\phi(z)=w-F(w)$ and then $\text{Im}\, \phi(z)=\text{Im}\, w-\text{Im}\, F(w)\leq 0$ since $\text{Im}\, F(w) > \text{Im}\, w$. 
If $z\notin F(\mathbb{C}^+)$ either $z\in F_1(U_1)$ or $z\in F_2(U_2)$. Assume without loss of generality that $z\in F_1(U_1)$. Then there is $w\in U_1\setminus \mathbb{C}^+$ such that $\phi(z)=w-F_1(w)=w-z$ and then since $\text{Im}\, z>0$ and $\text{Im}\, w\leq0$ we have $\text{Im}\, \phi(z)\leq 0.$  
This proves that $\phi(z)=F^{-1}(z)-z$ maps $\mathbb{C}^+$ into $\mathbb{C}^- \cup \mathbb{R}$ and therefore $\mathbf{b}^\rho_\alpha$ is freely infinitely divisible.

We prove that $\mathbf{b}^\rho_\alpha\notin\mathcal{UI}$. The sheets $U_1$ and $U_2$ have an intersection on $\mathbb{C}^-$. If we take a $w$ from the intersection with sufficiently small $|w|$, then we can prove $F_1(w) \neq F_2(w)$ and $F_1(w), F_2(w) \in \mathbb{C}^+$. 
This means $F^{-1}$ is not univalent in $\mathbb{C}^+$. 
\end{proof}

The case $(\alpha,\rho) =(1,\frac{1}{2})$ corresponds to the Cauchy distribution, which is freely infinitely divisible. 
Now we show that $\textbf{b}_\alpha^\rho$ does not belong to $ID(\boxplus)$ for the remaining parameters. First we consider the case $\alpha \geq 1$. 

Let $\mu$ be $\boxplus$-infinitely divisible and $\mu_t:=\mu^{\boxplus t}$. 
For $s \leq t$, a subordination function $\omega_{s,t}: \mathbb{C}^+\to\mathbb{C}^+$ exists so that it satisfies $F_{\mu_s} \circ \omega_{s,t} = F_{\mu_t}$. This relation is equivalent to 
\begin{equation}\label{eq1}
F_{\mu_t}(z)= \frac{t/s}{t/s-1}\omega_{s,t}(z)-\frac{z}{t/s-1}.
\end{equation} 
 It is proved in Theorem 4.6 of \cite{Bel3} that $\omega_{s,t}$ and hence $F_{\mu_t}$ extends to a continuous function from $\mathbb{C}^+\cup \mathbb{R}$ into itself. 
 % and satisfies the inequality 
%$$
%|\omega_{s,t}(z_1)-\omega_{s,t}(z_2)| \geq \frac{1}{2}|z_1 - z_2|, ~~~z_1,z_2 \in \mathbb{C}^+ \cup \mathbb{R}. 
%$$ 
%Therefore, $F_{\mu_t}$ also extends to a continuous function from $\mathbb{C}^+\cup \mathbb{R}$ into itself. 
%Taking the limit $s \to 0$ in (\ref{eq1}), we get 
%$$
%|F_{\mu_t}(z_1)-F_{\mu_t}(z_2)| \geq \frac{1}{2}|z_1 - z_2|, ~~~z_1,z_2 \in \mathbb{C}^+ \cup \mathbb{R}, 
%$$ 
%so that $F_{\mu_t}$ is univalent. 

%Since $F_\mu$ with $\mu$ $\boxplus$-infinitely divisible extends to a continuous function from $\mathbb{C}^+ \cup \mathbb{R}$ into itself, we have the following. 

\begin{proposition} Boolean stable laws $\mathbf{b}_\alpha^\rho$ are not $\boxplus$-infinitely divisible in the following cases: 

(1) $\alpha \in (1, 2]$,  

(2) $\alpha = 1$ and $\rho \neq \frac{1}{2}$. 
\end{proposition}
\begin{proof} 
In these cases $F_{\textbf{b}_\alpha^\rho}(0) = \infty$, and hence $F_{\textbf{b}_\alpha^\rho}$ does not extend to a continuous function on $\mathbb{C}^+ \cup \mathbb{R}$. 
\end{proof}

It only remains to discuss the case $\frac{1}{2}<\alpha<1$. 

\begin{proposition} Let $\frac{1}{2}<\alpha<1$, then $\mathbf{b}^\rho_\alpha$ is not $\boxplus$-infinitely divisible unless $\alpha\leq\frac{2}{3}$ and $2-\frac{1}{\alpha}\leq\rho\leq\frac{1}{\alpha}-1$.
\end{proposition}

\begin{proof}
Let $\phi:=\alpha\rho\pi$ and $F(z):=F_{\normalfont \textbf{b}_\alpha^\rho}(z)$. Suppose $\phi<(2\alpha-1)\pi$. 
A key quantity here is the angle $\theta_3:=\frac{\phi+\pi}{\alpha}$. It holds that $\pi<\theta_3 < 2\pi$ and 
\begin{eqnarray}
F(re^{i\theta_3})&=&re^{i\theta_3}+e^{i\phi}r^{1-\alpha} e^{i(1-\alpha)\theta_3}\\ \notag
&=& (r-r^{1-\alpha})e^{i\theta_3}.
\end{eqnarray}
Hence $F$ is not injective on $l_{\theta_3}$ and $F(l_{\theta_3}) = \{re^{i\theta_3}: r \geq -\alpha(1-\alpha)^{\frac{1-\alpha}{\alpha}}\}$. On the other hand, $F(re^{i\phi/\alpha}) = (r +r^{1-\alpha})e^{i\phi/\alpha}$ and then $F$ is bijective from $l_{\phi/\alpha}$ onto itself. Let us prove that the map $F$ is univalent in $\mathbb{C}_{(\phi/\alpha,\theta_3)}$. 

We fix $\theta\in (\phi/\alpha,\theta_3)$. Then $ \theta > \phi+ (1-\alpha)\theta > \theta -\pi$ and hence $F(re^{i\theta}) =re^{i\theta} + r^{1-\alpha}e^{i(\phi+(1-\alpha)\theta)} \in \mathbb{C}_{(\phi+(1-\alpha)\theta, \theta)}$. Since $\phi+(1-\alpha)\theta > \phi/\alpha$, $F(l_\theta)$ does not intersect $l_{\phi/\alpha}$ except at $0$. Similarly to the proof of Proposition \ref{prop90}, $F$ is injective on $l_{\theta}$, to conclude $F$ is univalent in $\mathbb{C}_{(\phi/\alpha, \theta)}$. 

Therefore, $F$ maps $\mathbb{C}_{(\phi/\alpha, \theta_3)}$ onto $\mathbb{C}_{(\phi/\alpha,\theta_3)} \setminus \{rz_0:  r \in [0,1] \}$, where $z_0 = -\alpha(1-\alpha)^{\frac{1-\alpha}{\alpha}}e^{i\theta_3}$. 
We can find $0< r_1<r_2$ such that $F(r_1e^{i\theta_3})=F(r_2e^{i\theta_3})=\frac{z_0}{2} \in \mathbb{C}^+$, so that 
 $$
 \lim_{\epsilon \rightarrow 0^+}F^{-1}\left(\frac{z_0}{2}+\epsilon\right)=r_1e^{i\theta_3}
 $$
but
$$
\lim_{\epsilon \rightarrow 0^-}F^{-1}\left(\frac{z_0}{2}+\epsilon \right)=r_2e^{i\theta_3}. 
$$
Hence the Voiculescu transform does not extend to $\mathbb{C}^+$ continuously. 

The case $\phi>(1-\alpha)\pi$ is analogous by choosing $\theta_4=\frac{\phi-\pi}{\alpha} \in (-\pi,0)$ instead of $\theta_3$.
 \end{proof}

\section{Classical infinite divisibility of positive Boolean stable laws}\label{S4}

For the convenience of the reader we start recalling basic facts on classical infinite divisibility and completely monotonic functions. 
For $\mu,\nu\in\mathcal{P}$, let $\mu\ast\nu$ denote the classical convolution of $\mu$ and $\nu$. 

%That is $$\mu*\nu(-\infty,a)=\int_\infty^a{ ??    }$$
\begin{definition} A probability measure $\mu$ is said to be \textit{$\ast$-infinitely divisible} if for each $n\in\mathbb{N}$ there exist $\mu_n$ such that $\mu=\mu_n^{\ast n}.$ We will denote by $ID(\ast)$ the set of $\ast$-infinitely divisible measures. 
\end{definition}

\begin{definition} A function $f: (0,\infty) \to \mathbb{R}$ is said to be \textit{completely monotonic (c.m.)}, if it possesses derivatives $f^{(n)}(x)$ for all $ n = 0, 1,2, 3,\dots$ and if
$$
(-1)^nf^{(n)}(x)\geq 0, ~~x > 0,~~ n=0,1,2,\dots.
$$
\end{definition}

The importance of complete monotonicity in relation to classical infinite divisibility is given by the following well known criterion (see e.g.\ \cite{Fell}).
\begin{theorem}
Let $\mu$ be a probability distribution on $[0,\infty)$ with absolutely continuous density w.r.t.\ the Lebesgue measure. If the density is c.m.\ then the measure is in $ID(\ast)$.
\end{theorem}
 
We will use the following well known properties of c.m. functions.

\begin{proposition}\label{cm10}
(1) If $f(x)$ and $g(x)$ are c.m., then $f(x)g(x)$, $f(x)+g(x)$ are also c.m.  

(2) Let $f(x)$ be c.m.\ and $h(x)$ be a positive differentiable function. If $h'(x)$ is c.m.\ then $f(h(x))$ is also a c.m.\ function.   
\end{proposition}
The density of the Boolean stable law on $[0,\infty)$ can be  calculated as 
$$
\textbf{b}_\alpha^1 (dx)= \frac{1}{\pi} \frac{\sin(\alpha\pi) x^{\alpha-1}}{x^{2\alpha}+2\cos(\alpha\pi)x^\alpha +1}dx,~~~x>0,   
$$ 
for $\alpha \in (0,1)$. This density is c.m.\ for $\alpha \leq 1/2$ by basic properties of c.m.\ functions, see Proposition \ref{cm10}. Indeed, the function $h(x)=x^{2\alpha}+2 \cos(\alpha \pi)x^\alpha + 1$ has a c.m.\ derivative $h'(x)$, so that $\frac{1}{h(x)}$ is c.m. The function $\frac{x^{\alpha-1}}{h(x)}$ is the multiplication of two c.m.\ functions and hence is c.m. 
Thus we have shown the classical infinite divisibility together with the free infinite divisibility: 

\begin{theorem} The Boolean stable law ${\normalfont \bf b}_\alpha^1$ on $[0,\infty)$ belongs to $ID(\boxplus)\cap ID(\ast)$ for $\alpha \leq 1/2$. 
\end{theorem}
We note that $\textbf{b}_\alpha^{1}$ is a \textit{generalized gamma convolution} (GGC) which is a class of $ID(\ast)$. This fact follows from Theorem 2 of \cite{Bon}. 
 
More generally the complete monotonicity is still true for continuous Boolean convolutions of Boolean stable laws $\textbf{b}(\sigma)$ and hence one concludes the following:  
\begin{theorem} \label{mix2}
The continuous Boolean convolution ${\normalfont  \bf b}(\sigma) = \int^{\uplus}_{(0,\frac{1}{2}]} {\normalfont \bf b}_\alpha^1 \sigma(d\alpha)$ 
 is in $ID(\boxplus)\cap ID(\ast)$ for any non-negative finite measure $\sigma$ on $(0,\frac{1}{2}]$. 
\end{theorem}
\begin{proof} The density of $\textbf{b}(\sigma)$ is given by 
$$
\frac{1}{\pi}\cdot\frac{\int_{(0,\frac{1}{2}]}\sin(\alpha\pi)x^{-\alpha}\sigma(d\alpha)}{\left(x^{1/2}+\int_{(0,\frac{1}{2}]}\cos(\alpha\pi)x^{1/2-\alpha}\sigma(d\alpha)\right)^2 +\left(\int_{(0,\frac{1}{2}]}\sin(\alpha\pi)x^{1/2-\alpha}\sigma(d\alpha)\right)^2},~~~x>0. 
$$
We first consider the case where $\sigma$ is a finite sum of atoms: $\sigma= \sum_{k=1}^m \lambda_k \delta_{\alpha_k}$, $\alpha_k \in (0,\frac{1}{2}]$, $\lambda_k > 0$. Then the density of $\textbf{b}(\sigma)$ is given by $\frac{f(x)}{\pi g(x)}$, where 
\[
\begin{split}
&f(x) = \sum_{k=1}^m \lambda_k \sin(\alpha_k\pi) x^{-\alpha_k},\\
&g(x) = \left(x^{1/2}+\sum_{k=1}^m \lambda_k \cos(\alpha_k\pi) x^{1/2-\alpha_k}\right)^2  + \left(\sum_{k=1}^m \lambda_k \sin(\alpha_k\pi) x^{1/2-\alpha_k}\right)^2. 
\end{split}
\]

The function $f(x)$ is c.m. $g(x)$ is a finite sum of $\{x^{\beta}:\beta \in [0,1]\}$ with positive coefficients and hence $g'(x)$ is c.m. Therefore, the density $\frac{f(x)}{\pi g(x)}$ is c.m. 

For a general finite Borel measure $\sigma$, consider $\sigma_n:= \sum_{k=0}^{n-1} \sigma\left((\frac{k}{2n}, \frac{k+1}{2n}]\right) \delta_{\frac{2k+1}{4n}}$. 
%Then $\int_{(0,\frac{1}{2}]}f(\alpha)\sigma_n(d\alpha)$ converges to  $\int_{(0,\frac{1}{2}]}f(\alpha)\sigma(d\alpha)$ for any bounded continuous function $f$. 
Then $\frac{d^p}{dx^p}\int_{(0,\frac{1}{2}]}e^{i\alpha \pi}x^{1/2-\alpha}\sigma_n(d\alpha)$ converges to  $\frac{d^p}{dx^p}\int_{(0,\frac{1}{2}]}e^{i\alpha\pi}x^{1/2-\alpha}\sigma(d\alpha)$ locally uniformly on $(0,\infty)$ for any integer $p \geq 0$. This approximation argument shows that the denominator of the density of $\textbf{b}(\sigma)$ has c.m.\ derivative, and hence $\textbf{b}(\sigma)$ itself has c.m.\ density. 
\end{proof}

\section{Relations between Boolean, free and classical stable laws and multiplicative convolutions}\label{S2}
We prove a reproducing property of Boolean stable laws with respect to the free multiplicative convolution. 
The\textit{\ free multiplicative convolution} $\boxtimes$ of probability measures with bounded support on $[0,\infty)$ is defined as follows~\cite{Voi87}. Let $X_1,X_2$ be positive free random variables following distributions $\mu_1,\mu_2$, respectively. Then the free multiplicative convolution $\mu_1 \boxtimes \mu_2$ is defined by the distribution of $X^{1/2}YX^{1/2}$. 
More generally, free multiplicative convolution $\mu_1 \boxtimes \mu_2$ can be defined for any two probability measures $\mu_1$ and $\mu_2$ on $\mathbb{R}$, provided that one of them is supported on $[0,\infty)$; see \cite{BeVo93}. 

Let $\textbf{s}_\alpha^\rho$ be the free strictly stable law characterized by
$$
\phi_{\textbf{s}_\alpha^\rho}(z) = -e^{i\alpha\rho\pi}z^{1-\alpha}, ~~z \in \mathbb{C}^+, 
$$ 
$\alpha \in (0,1], \rho\in [0,1]$.  
Results of \cite{APA} and \cite{BePa} show a reproducing property for $\boxplus$-strictly stable laws $\mathbf{s}^\theta_{\alpha}$: 
\begin{equation}\label{eq13}
\mathbf{s}^\rho_{1/(1+t)}\boxtimes\mathbf{s}^1_{1/(1+s)}=\mathbf{s}^\rho_{1/(1+s+t)}
\end{equation}
for $\rho \in \{0,\frac{1}{2}, 1 \}$. This property is still unknown for general $\rho \in [0,1]$.  

We show that the same reproducing property holds for $\textbf{b}_\alpha^\rho$, including a new result for general $\rho \in [0,1]$.  

\begin{proposition} (1) For any $s,t>0$ and $\rho \in \{0,\frac{1}{2},1\}$,  $\mathbf{b}^\rho_{1/(1+t)}\boxtimes\mathbf{b}^1_{1/(1+s)} =\mathbf{b}^\rho_{1/(1+s+t)}$. 

(2) For any $s,t >0$ and $\rho \in [0,1]$, $\mathbf{s}^\rho_{1/(1+t)}\boxtimes\mathbf{s}^1_{1/(1+s)}$ is a $\boxplus$-strictly stable law with stability index $\frac{1}{1+s+t}$ and $\mathbf{b}^\rho_{1/(1+t)}\boxtimes\mathbf{b}^1_{1/(1+s)}$ is a $\uplus$-strictly stable law with stability index $\frac{1}{1+s+t}$. 
\end{proposition}

\begin{proof} 
(1) Since $\mathbb{B}_1$ is the Bercovici-Pata bijection from the Boolean convolution to the free one,  $\mathbb{B}_1(\textbf{b}_\alpha^\rho) = \textbf{s}_\alpha^\rho$. Moreover, $\mathbb{B}_1$ is a homomorphism with respect to $\boxtimes$~\cite{BN}, and hence (\ref{eq13}) implies    
$$
\mathbb{B}_1(\textbf{b}_{1/(1+t)}^\rho \boxtimes \textbf{b}_{1/(1+s)}^1) = \mathbb{B}_1(\textbf{b}_{1/(1+s+t)}^\rho), 
$$
and hence $\mathbf{b}^\rho_{1/(1+t)}\boxtimes\mathbf{b}^1_{1/(1+s)} =\mathbf{b}^\rho_{1/(1+s+t)}$ from the injectivity of $\mathbb{B}_1$.  

(2) 
Suppose $\mu$ or $\nu$ is supported on $[0,\infty)$. Then the relation  
\begin{equation}\label{eq23}
(\mu \boxtimes \nu)^{\uplus t} = D_{1/t} (\mu^{\uplus t} \boxtimes \nu^{\uplus t})   
\end{equation}
holds for $t>0$, while 
\begin{equation}
(\mu \boxtimes \nu)^{\boxplus t} = D_{1/t} (\mu^{\boxplus t} \boxtimes \nu^{\boxplus t})
\end{equation}
holds for $t\geq 1$; see Proposition 3.5 of \cite{BN}.  Using (\ref{eq23}), 
\[
\begin{split}
(\mathbf{b}^\rho_{1/(1+t)}\boxtimes\mathbf{b}^1_{1/(1+s)})^{\uplus 2}
&=D_{1/2}\left((\mathbf{b}^\rho_{1/(1+t)})^{\uplus 2}\boxtimes(\mathbf{b}^1_{1/(1+s)})^{\uplus 2}\right)\\
&=D_{1/2}\left(D_{2^{1+t}}(\mathbf{b}^\rho_{1/(1+t)})\boxtimes D_{2^{1+s}}(\mathbf{b}^1_{1/(1+s)})\right)\\
&=D_{2^{1+t+s}}(\mathbf{b}^\rho_{1/(1+t)}\boxtimes\mathbf{b}^1_{1/(1+s)}). 
\end{split}
\] 
By definition, $\mathbf{b}^\rho_{1/(1+t)}\boxtimes\mathbf{b}^1_{1/(1+s)}$ is $\uplus$-strictly stable with stability index $\frac{1}{1+s+t}$. The proof for the free case is similar. 
\end{proof}

This reproducing property combined with the fact $\mathbf{b}^\rho_{1/2} \in ID(\boxplus)$ shows that some Boolean stable laws are in $ID(\boxplus)$, giving an alternative look at Theorem \ref{T1}. 

\begin{corollary}
The Boolean stable law $\mathbf{b}_{1/n}^\rho$ is freely infinitely divisible if $n \in \{2,3,4,\cdots \}$ and 
$\rho \in \{0,\frac{1}{2},1\}$.
\end{corollary}
\begin{proof}
Since $\mathbf{b}^\rho_{1/2} \in ID(\boxplus)$ and $\mathbf{b}^1_{1/2}$ is free regular by Proposition \ref{T2}, $\mathbf{b}^\rho_{1/2}\boxtimes \mathbf{b}^1_{1/2}=\mathbf{b}^\rho_{1/3}$ is also freely infinitely divisible by Theorem \ref{closure}. In a similar way, by induction, we conclude that $\textbf{b}^\rho_{1/n}$ is $\boxplus$-infinitely divisible for $n \geq 2$. 
\end{proof}

We note here an interplay among free, Boolean and classical stable laws. Let $\textbf{s}_\alpha$ be a free strictly stable law on $[0,\infty)$ and $\textbf{n}_\alpha$ be a classical strictly stable law defined by 
\[
\begin{split}
&\phi_{\textbf{s}_\alpha} (z) = -e^{i\alpha \pi}z^{1-\alpha},~~~z \in \mathbb{C}^+, \\
&\int_{0}^\infty e^{-xz} \textbf{n}_\alpha(dx)= e^{-z^\alpha}, ~~z >0. 
\end{split}
\]
Let $\circledast$ denote the classical multiplicative convolution: $$\int_{[0,\infty)}f(x)(\mu \circledast \nu)(dx) = \int_{[0,\infty)^2}f(xy)\mu(dx)\nu(dy)$$ for any bounded continuous function $f$ on $[0,\infty)$.  
If $\check{\textbf{s}}_\alpha$ and $\check{\textbf{n}}_\alpha$ respectively denote the push-forwards of $\textbf{s}_\alpha$ and $\textbf{n}_\alpha$ by the map $x \mapsto 1/x$, then $\textbf{s}_\alpha \boxtimes \check{\textbf{s}}_\alpha = \textbf{n}_\alpha \circledast \check{\textbf{n}}_\alpha$ for $\alpha \in (0,1]$ as proved in  \cite{BePa}.  Moreover, this coincides with a Boolean stable law:  
$$
\textbf{s}_\alpha \boxtimes \check{\textbf{s}}_\alpha = \textbf{n}_\alpha \circledast \check{\textbf{n}}_\alpha = \textbf{b}_\alpha^1.  
$$
This fact follows from the calculation of the density shown in Proposition A4.4 of \cite{BePa}.

\bibliographystyle{amsplain}

\begin{thebibliography}{100}
\bibitem{AH1} O.\ Arizmendi and T.\ Hasebe, On a class of explicit Cauchy-Stieltjes transforms related to monotone stable and free Poisson laws, to appear in Bernoulli. arXiv:1108.3438 
\bibitem{AH2} O.\ Arizmendi and T.\ Hasebe, Semigroups related to additive and multiplicative, Boolean and free convolutions. arXiv:1105.3344
\bibitem{AHS} O.\ Arizmendi, T.\ Hasebe and N.\ Sakuma, On the law of free subordinators. arXiv:1201.0311 
\bibitem{APA} O.\ Arizmendi and V.\ P\'erez-Abreu, The S-transform for symmetric probability measures with unbounded supports, Proc.\ Amer.\ Math.\ Soc.\ \textbf{137} (2009), 3057--3066. 
\bibitem{Bel3} S.T.~Belinschi and H.~Bercovici, Partially defined semigroups relative to multiplicative free convolution, Int.\ Math.\ Res.\ Notices, No.\ 2 (2005), 65--101.  
\bibitem{BBLS} S.T.~Belinschi, M.\ Bo\.{z}ejko, F.\ Lehner and  R.\ Speicher, The normal distribution is $\boxplus$-infinitely divisible, Adv.\  Math.\ \textbf{226}, No.\ 4 (2011), 3677--3698. 

%\bibitem{M S} K.S. Miller, S. G. Samko, \emph{Completely monotonic functions} Integr. Transf. and Spec. Funct.2001, Vol. 12, No 4, 389-402 
\bibitem{BN} S.T.~Belinschi and A.~Nica, On a remarkable semigroup of homomorphisms with respect to free multiplicative convolution, Indiana Univ.\ Math.\ J.\ \textbf{57}, No.\ 4 (2008), 1679--1713.  
\bibitem{BePa} H.~Bercovici and V.~Pata, Stable laws and domains of attraction in free probability theory (with an appendix by Philippe Biane), Ann.\ of Math.\ (2) \textbf{149}, No.\ 3 (1999), 1023--1060. 
\bibitem{BeVo93} H.~Bercovici and D.~Voiculescu, L\'{e}vy-Hin\v{c}in type theorems for multiplicative and additive free convolution, Pacific J.\ Math.\ \textbf{153}, No.\ 2 (1992), 217--248. 
\bibitem{Be-Vo} H.~Bercovici and D.~Voiculescu, Free convolution of measures with unbounded support, Indiana Univ.\ Math.\ J.\ \textbf{42}, No.\ 3 (1993), 733--773. 
\bibitem{Bon} L.\ Bondesson, A general result on infinite divisibility, Ann.\ Probab.\ \textbf{7}, No.\ 6 (1979), 965--979. 
 \bibitem{Buc} R.B.\ Burckel, \textit{An introduction to classical complex analysis}, vol.\ 1, Academic Press, New York, 1979. 
 \bibitem{Fell} W.\ Feller, An Introduction to Probability Theory and its Applications, vol.II. John Wiley and Sons, 1971. 
 \bibitem{NS97} A.\ Nica and R.\ Speicher, On the multiplication of free $n$-tuples of noncommutative
random variables, Amer.\ J.\ Math.\ \textbf{118} (1996), 799--837.
\bibitem{PAS} V.\ P\'{e}rez-Abreu and N.\ Sakuma, Free generalized gamma convolutions, Electron.\ Commun.\ Probab.\ \textbf{13 }(2008) 526--539. 
\bibitem{S-W} R.~Speicher and R.~Woroudi, Boolean convolution, in Free Probability Theory, Ed.~D.~Voiculescu, Fields Inst.\ Commun.\ \textbf{12} (Amer.\ Math.\ Soc., 1997), 267--280. 
\bibitem{Voi86} D.\ Voiculescu, Addition of certain non-commutative random variables, J.\ Funct.\ Anal. \textbf{66} (1986), 323--346.
\bibitem{Voi87} D.\ Voiculescu, Multiplication of certain noncommuting random variables, J.\ Operator Theory \textbf{18} (1987), 223--235.  
%\bibitem{Wang1} J.-C.\ Wang, Limit laws for Boolean convolutions, Pacific J.\ Math.\ \textbf{237}, No.\ 2 (2008), 349--371. 
\end{thebibliography}

\end{document}